\documentclass[a4paper,11pt]{amsart}
\usepackage{amssymb}
\usepackage{cmap}
\usepackage{hyperxmp}
\usepackage[pdfdisplaydoctitle=true,
            colorlinks=true,
            urlcolor=blue,
            citecolor=blue,
            linkcolor=blue,
            pdfstartview=FitH,
            pdfpagemode= UseNone,
            bookmarksnumbered=true]{hyperref}
\usepackage{todonotes}

\theoremstyle{plain}
\newtheorem{thm}{Theorem}[section]
\newtheorem{cor}[thm]{Corollary}
\newtheorem{lem}[thm]{Lemma}
\newtheorem{prop}[thm]{Proposition}

\theoremstyle{definition}
\newtheorem{defn}[thm]{Definition}

\theoremstyle{definition}
\newtheorem{rem}[thm]{Remark}

\newtheorem{expl}[thm]{Example}
\newtheorem*{ass*}{Assumption}

\numberwithin{equation}{section}

\newcommand{\NN}{\mathbb{N}} 
\newcommand{\ZZ}{\mathbb{Z}} 
\newcommand{\RR}{\mathbb{R}} 
\newcommand{\CC}{\mathbb{C}} 

\newcommand{\Circle}{\mathrm{S}^{1}}    

\newcommand{\Diff}{\mathrm{Diff}}       
\newcommand{\Imm}{\mathrm{Imm}}         
\newcommand{\D}[1]{\mathcal{D}^{#1}}    

\newcommand{\HRd}[1]{H^{#1}(\Circle,\mathbb{R}^{d})}    
\newcommand{\LRd}{L^{2}(\Circle,\mathbb{R}^{d})}        

\newcommand{\norm}[1]{\left\Vert#1\right\Vert}          
\newcommand{\abs}[1]{\left\vert#1\right\vert}           
\newcommand{\set}[1]{\left\{#1\right\}}                 
\newcommand{\op}[1]{\mathbf{op}\left(#1\right)}         
\newcommand{\llangle}{\langle\!\langle}                 
\newcommand{\rrangle}{\rangle\!\rangle}                 
\newcommand{\sprod}[2]{\left\langle #1,#2 \right\rangle}

\newcommand{\ud}{\,\mathrm{d}}

\hypersetup{
    pdftitle={Fractional Sobolev metrics on spaces of immersed curves}, 
    pdfauthor={M. Bauer, M. Bruveris and B. Kolev}, 
    pdfsubject={MSC 2010: 58D05, 35Q35}, 
    pdfkeywords={Sobolev metrics of fractional order}, 
    pdflang=EN 
    }
\begin{document}

\title{Fractional Sobolev metrics on spaces of immersed curves}

\author{Martin Bauer}
\address{Faculty for Mathematics, Florida State University, USA}
\email{bauer@math.fsu.edu}

\author{Martins Bruveris}
\address{Department of Mathematics, Brunel University London,
  Ux\-bridge, UB8 3PH, United Kingdom}
\email{martins.bruveris@brunel.ac.uk}

\author{Boris Kolev}
\address{Aix Marseille Universit\'{e}, CNRS, Centrale Marseille, I2M, UMR 7373, 13453 Marseille, France}
\email{boris.kolev@math.cnrs.fr}

\subjclass[2010]{58D05, 35Q35}
\keywords{Sobolev metrics of fractional order} %

\date{\today}

\begin{abstract}
  Motivated by applications in the field of shape analysis, we study reparametrization invariant, fractional order Sobolev-type metrics on the space of smooth regular curves $\Imm(\Circle,\RR^d)$ and on its Sobolev completions $\mathcal{I}^{q}(\Circle,\mathbb{R}^{d})$. We prove local well-posedness of the geodesic equations both on the Banach manifold $\mathcal{I}^{q}(\Circle,\mathbb{R}^{d})$ and on the Fr\'{e}chet-manifold $\Imm(\Circle,\RR^d)$ provided the order of the metric is greater or equal to one. In addition we show that the $H^s$-metric induces a strong Riemannian metric on the Banach manifold $\mathcal{I}^{s}(\Circle,\mathbb{R}^{d})$ of the same order $s$, provided $s>\frac 32$. These investigations can be also interpreted as a generalization of the analysis for right invariant metrics on the diffeomorphism group.
\end{abstract}

\maketitle

\setcounter{tocdepth}{2}
\tableofcontents

\section{Introduction}
\label{sec:introduction}

The interest in Riemannian metrics on infinite-dimensional manifolds is fueled by their connections to mathematical physics and in particular fluid dynamics. It was Arnold who discovered in 1966 that the incompressible Euler equation, which describes the motion of an ideal fluid, has an interpretation as the geodesic equation on an infinite-dimensional manifold; the manifold in question is the group of volume-preserving diffeomorphisms equipped with the $L^{2}$-metric. Since then many other PDEs in mathematical physics have been reinterpreted as geodesic equations. Examples include Burgers' equation, which is the geodesic equation of the $L^{2}$-metric on the group of all diffeomorphisms of the circle, $\Diff(\Circle)$, and the Camassa--Holm equation \cite{CH1993}, the geodesic equation of the $H^1$-metric \cite{Kou1999} on the same group. Interestingly, geodesic equations corresponding to fractional orders in the Sobolev scale have also found applications in physics: Wunsch showed that the geodesic equation of the homogenous $\dot{H}^{1/2}$-metric on $\Diff(\Circle)$ is connected to the Constantin--Lax--Majda equation \cite{CLM1985,Wun2010,EKW2012}, which itself is a simplified model of the vorticity equation.

The geometric interpretation of a PDE as the geodesic equation enables one to show local well-posedness of the PDE. This was done first by Ebin and Marsden~\cite{EM1970} for the Euler equation. Using a similar method Constantin and Kolev showed in~\cite{CK2003} that the geodesic equation of Sobolev $H^n$-metrics on $\Diff(\Circle)$ with integer $n\geq 1$ is locally well-posed. In~\cite{EK2014} this was extend by Escher and Kolev to the Sobolev $H^r$-metrics of fractional order $r \geq \tfrac 12$. Similar results were shown for the diffeomorphism group of compact manifolds by Shkoller in \cite{Shk1998,Shk2000} and by Preston and Misiolek in~\cite{MP2010}. Fractional metrics on $\Diff(\RR^d)$ have been studied in \cite{BEK2015} by Bauer, Escher and Kolev. The local well-posedness of the geodesic equation for fractional order metrics on the diffeomorphism group of a general manifold $M$ remains an open problem.

In this paper we study the local well-posedness of a family of PDEs that arise as geodesic equations on the space $\Imm(\Circle,\RR^d)$ of immersed curves. To be precise $\Imm(\Circle,\RR^d)$ consists of smooth, closed curves with nowhere vanishing derivatives. We can regard the diffeomorphism group
\begin{equation*}
  \Diff(\Circle) \subseteq \Imm(\Circle,\Circle)
\end{equation*}
as an open subset of the space of immersions. If we replace $\Circle$ on the right hand side by $\RR^d$ we obtain the space of curves. The PDE
\begin{equation*}
  c_{tt} = -\langle D_{s} c_{t}, D_{s} c \rangle c_{t} - \langle c_{t}, D_{s} c_{t} \rangle D_{s} c - \frac 12 \abs{c_{t}}^{2} D_{s}^{2} c\,,
\end{equation*}
where $D_{s} = \frac 1{\abs{c'}} \partial_\theta$ and $c=c(t,\theta)$ is a time-dependent curve, is the geodesic equation for the $L^{2}$-metric
\begin{equation*}
  G_{c}(h,k) = \int_{\Circle} \langle h, k \rangle \abs{c'} \ud \theta\,.
\end{equation*}
This is a (weak) Riemannian metric on $\Imm(\Circle,\RR^d)$. The weight $\abs{c'}$ in the integral makes the metric invariant under the natural $\Diff(\Circle)$-action and leads to the appearance of arc length derivatives $D_{s}$ in the geodesic equation. This PDE can be seen as a generalization of the geodesic equation on $\Diff(\Circle)$,
\begin{equation*}
  \varphi_{tt} = -2 \frac{\varphi_{t} \varphi_{tx}}{\varphi_x}\,,
\end{equation*}
which becomes, when written in terms of the Eulerian velocity $u = \varphi_{t} \circ \varphi^{-1}$, Burgers' equation,
\begin{equation*}
  u_{t} = -2u_x u\,.
\end{equation*}
While the behaviour of Burgers' equation is well-known, to our knowledge, nothing is known about the local well-posedness of the $L^{2}$-geodesic equation on the space of curves.

The situation improves when we add higher derivatives to the Riemannian metric. By doing so, we get the $H^n$-metric
\begin{equation*}
  G_{c}(h,k) = \int_{\Circle} \langle A_{c} h, k \rangle \abs{c'}\ud \theta\, , \qquad A_{c} = \sum_{j=0}^n (-1)^{j} \alpha_{j} D_{s}^{2j}\,,
\end{equation*}
where $\alpha_{j} \geq 0$ are constants. The corresponding geodesic equation is
\begin{align*}
  (A_{c} c_{t})_{t} & = -\langle D_{s} c_{t}, D_{s} c \rangle A_{c} c_{t} - \langle A_{c} c_{t}, D_{s} c_{t} \rangle D_{s} c - W(c,c_{t}) D_{s}^{2} c\,,
\end{align*}
with
\begin{align*}
  W(c,c_{t}) & = \frac 12 \abs{c_{t}}^{2} + \frac{1}{2} \sum_{j=1}^n \sum_{k=1}^{2j-1} (-1)^{k+1} \langle D_{s}^{2j-k}c_{t}, D_{s}^{k}c_{t} \rangle \,.
\end{align*}
Interestingly, the behaviour of the geodesic equation is better understood in the seemingly more complicated case where $n \geq 1$. For $n=1$ the PDE is locally but generally not globally well-posed \cite{MM2007}, while for $n \geq 2$ the PDE has solutions that are global in time~\cite{BMM2014}.

Inspired by related work \cite{BEK2015, EK2014} on geodesic equations on the diffeomorphism group we consider $H^r$-metrics of non-integer order $r$, i.e.
\begin{equation*}
  G_{c}(h,k) = \int_{\Circle} \langle A_{c} h, k \rangle \abs{c'} \ud \theta\,,
\end{equation*}
with $A_{c}$ being, for each fixed curve $c$, a Fourier multiplier of order $2r$; the precise assumptions made on $A_{c}$ are described in Section~\ref{sec:metrics}. The geodesic equation takes the form
\begin{equation*}
  (A_{c} c_{t})_{t} = -\langle D_{s} c_{t}, D_{s} c \rangle A_{c} c_{t} - \langle A_{c} c_{t}, D_{s} c_{t} \rangle D_{s} c -
  (w(c,c_{t}) + w_0(c,c_{t}) D_{s}^{2} c\,,
\end{equation*}
where $w(c,c_{t})$ and $w_0(c,c_{t})$ are expressions that are defined in Theorem~\ref{thm:geodesic-equations}. It is a nonlinear evolution equation of second order in $t$ and of order $2s$ in $\theta$; the right hand side is quadratic in $c_{t}$ and highly nonlinear in $c$. For non-integer $s$, both $A_{c} c_{t}$ and $w(c,c_{t})$ are nonlocal functions of both $c$ and $c_{t}$. While not immediately obvious, we show in Example~\ref{ex:metrics-constant-coeff} that it contains the geodesic equations for the $L^{2}$-metric and the $H^n$-metrics as special cases.

Our contribution is to show in Corollary~\ref{cor:Hq-local-well-posedness} that the geodesic equation for the $H^r$-metric is locally well-posed in the Sobolev space $H^q$ for $q \geq 2r$ and $r \geq 1$.

\subsection*{Connections to shape analysis}

Throughout the article we will assume that the operator $A_{c}$ defining the metric is equivariant with respect to the action of the diffeomorphism group $\Diff(\Circle)$, i.e.,
\begin{equation*}
  A_{c\circ\varphi}(h \circ\varphi)=\left(A_{c}h\right)\circ\varphi\,, \quad \forall \varphi\in\Diff(\Circle)\,.
\end{equation*}
This is equivalent to requiring that the Riemannian metric $G$ is invariant under the action of $\Diff(\Circle)$, meaning
\begin{equation*}
  G_{c \circ \varphi}(h \circ \varphi, k \circ \varphi) = G_{c}(h,k)\,, \quad \forall \varphi\in\Diff(\Circle)\,.
\end{equation*}
This assumption is necessary in order to apply the class of Sobolev metrics in shape analysis.

The mathematical analysis of shapes has been the focus of intense research in recent years
\cite{kendall_shape:84,le-kendall:93,dryden-mardia_book:98,You2010,srivastava-klassen-book:2016} and has found applications in fields such as image analysis, computer vision, biomedical imaging and functional data analysis. An important class of shapes are outlines of planar objects. Mathematically, these shapes can be represented by equivalence classes of parametrized curves modulo the reparametrization group $\Diff(\Circle)$. This yields the following geometric picture,
\begin{align*}
  \pi : \operatorname{Imm}(\Circle,\RR^d) \to \Imm(\Circle,\RR^{d})/\Diff(\Circle)\,.
\end{align*}
A key step in shape analysis is to define an efficiently computable distance function between shapes in order to measure similarity between shapes. This distance function can be the geodesic distance induced by a Riemannian metric. Since it is difficult to work with the quotient $\Imm(\Circle,\RR^{d})/\Diff(\Circle)$ directly, the standard approach is instead to define a Riemannian metric on $\Imm(\Circle,\RR^d)$, that is $\Diff(\Circle)$-invariant.

Such a metric can then induce a Riemannian metric on the quotient space, making $\pi$ a Riemannian submersion; this metric would be given by the formula
\[
G_{\pi(c)}(u, u) = \inf_{T_c \pi.h = u} G_c(h,h)\,.
\]
Hence we will only look at $\Diff(\Circle)$-invariant metrics on $\Imm(\Circle,\RR^d)$ in this paper.


The arguably simplest invariant metric is the $L^{2}$-metric, corresponding to the operator $A_{c}=\operatorname{Id}$. However, its induced geodesic distance function is identically zero, both, on the space of parametrized curves as well as the quotient space of unparametrized curves; this was shown by Michor and Mumford \cite{MM2005,MM2006}, see also \cite{BBHM2012}. This means that between any two curves there exists a path of arbitrary short length. This property makes the metric ill-suited for most applications in shape analysis, because a notion of distance between shapes is one of the basic tools there. As a consequence higher order metrics, mostly of integer order, were studied and used successfully in applications; see \cite{MM2007,Klassen2004,Jermyn2011,You1998,Sundaramoorthi2008,Sundaramoorthi2011}. 
For an overview on various metrics on the shape space of parametrized and unparametrized curves see
\cite{MM2007,BBM2014a,BBM2016}.
As an example we have included an optimal deformation between two unparametrized curves with respect to a second order Sobolev metric in Figure~\ref{fig:geodesic}.
\begin{figure}
  \includegraphics[width=0.8\linewidth]{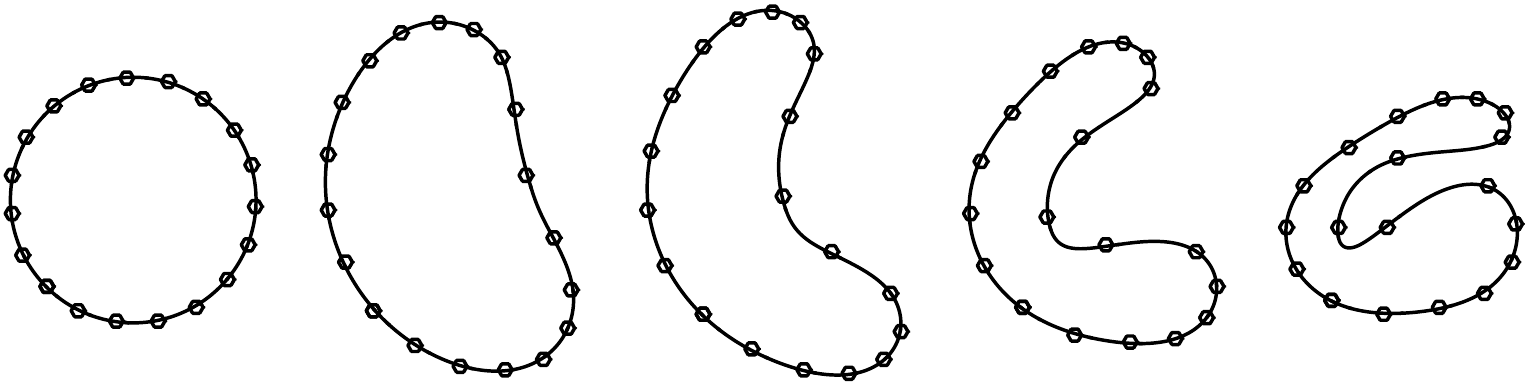}
  \caption{Example of a geodesic between two shapes with respect to an $H^{2}$-metric \cite{BBHM2016}.}
  \label{fig:geodesic}
\end{figure}
Fractional order metrics have been briefly mentioned \cite[Section 3]{MYS2008}. However, so far, an analysis of the well-posedness of the corresponding geodesic equation was missing. This is the question we will focus on in this paper.

\section{Parametrized and unparametrized curves}
\label{sec:curves}

In this article we consider the space of smooth regular curves with values in $\RR^d$
\begin{align*}
  \Imm(\Circle,\RR^{d}) := \set{c\in C^{\infty}(\Circle,\RR^{d}): \abs{c^{\prime}}\neq 0}\;.
\end{align*}
As an open subset of the Fr\'{e}chet space $C^\infty(\Circle,\RR^d)$, it is a Fr\'{e}chet manifold. Its tangent space at any curve $c$ is the vector space of smooth functions:
\begin{align*}
  T_{c}\Imm(\Circle,\RR^d)=C^{\infty}(\Circle,\RR^d)\;.
\end{align*}
The group of smooth diffeomorphisms of the circle
\begin{align*}
  \Diff(\Circle) := \set{\varphi\in C^{\infty}(\Circle,\Circle): \abs{\varphi'} > 0}
\end{align*}
acts on the space of regular curves via composition from the right:
\begin{align*}
  \Imm(\Circle,\RR^d)\times \Diff(\Circle)\rightarrow \Imm(\Circle,\RR^d),\qquad (c,\varphi)\mapsto c\circ\varphi\;.
\end{align*}
Taking the quotient with respect to this group action, we obtain the \emph{shape space}
of un-parameterized curves
\begin{align*}
  B_i(\Circle,\RR^d) := \Imm(\Circle,\RR^d)/\Diff(\Circle)\;.
\end{align*}
Note that the action of $\Diff(\Circle)$ on $\Imm(\Circle,\RR^d)$ is \emph{not free}, and thus that the quotient space
$B_i(\Circle,\RR^d)$ is not a manifold, but only an orbifold with finite isotropy groups.
A way to overcome this difficulty is to consider the slightly smaller space of \emph{free immersions} $\Imm_{\operatorname{f}}(\Circle,\RR^d)$, i.e., those immersions upon which $\Diff(\Circle)$ acts freely. This space is an
open and dense subset of $\Imm(\Circle,\RR^d)$ and the corresponding quotient space
\begin{align*}
  B_{i,\operatorname{f}}(\Circle,\RR^d) := \Imm_{\operatorname{f}}(\Circle,\RR^d)/\Diff(\Circle)\;
\end{align*}
is again a Fr\'{e}chet manifold, see \cite{CMM1991}.

\section{Riemannian metrics on immersions}
\label{sec:metrics}

Let $G$ be a Riemannian metric on $\Imm(\Circle,\RR^d)$. Motivated by applications in the field of shape analysis we require $G$ to be invariant with respect to the diffeomorphism group $\Diff(\Circle)$, i.e.,
\begin{equation}\label{eq:invariance}
  G_{c}(h,k) = G_{c\circ\varphi}(h\circ\varphi,k\circ\varphi),\qquad \forall \varphi \in \Diff(\Circle)\,.
\end{equation}
This invariance is a necessary assumption for $G$ to induce a Riemannian metric on the shape space $B_{i,f}(\Circle,\RR^d)$, such that the projection map is a Riemannian submersion.

We assume that the metric is given in the form
\begin{equation}\label{eq:metric}
  G_{c}(h,k) = \int_{\Circle}\sprod{A_{c} h}{k}\ud s ,
\end{equation}
where $A_{c}:C^\infty(\Circle,\RR^d) \to C^\infty(\Circle,\RR^d)$ is a continuous linear operator that depends on the foot point $c$. Associated to the metric is a map
\begin{equation*}
  \check G: T\Imm(\Circle,\RR^{d}) \to T^{*}\Imm(\Circle,\RR^{d}).
\end{equation*}
In terms of $A_{c}$ we have $\check G_{c} = A_{c} \otimes \ud s$. If $G$ is $\Diff(\Circle)$-invariant, then the family of operators $A_{c}$ has to be $\Diff(\Circle)$-equivariant,
\begin{equation*}
  A_{c\circ\varphi}(h\circ\varphi)=A_{c}(h)\circ\varphi,\qquad \forall \varphi \in \Diff(\Circle)\,.
\end{equation*}

\begin{expl}[Integer order Sobolev metrics]
  Motivated by their applicability in shape analysis \cite{YMSM2008,SYM2007,MYS2008,MM2007} important examples  are Sobolev metrics with \emph{constant coefficients},
  \begin{equation}\label{eq:integer-order-metric}
    G^{n}_{c}(h,k) = \sum_{j=0}^{n} \alpha_{j} \int_{\Circle} \sprod{D_{s}^{j} h}{D_{s}^{j}k} \ud s\,,
  \end{equation}
  and \emph{scale-invariant} Sobolev metrics,
  \begin{equation}\label{eq:scale-integer-metric}
    \widetilde G^{n}_{c}(h,k) = \sum_{j=0}^{n} \alpha_{j} \int_{\Circle} \ell_{c}^{2j-3} \sprod{D_{s}^{j} h}{D_{s}^{j}k} \ud s\,,
  \end{equation}
  with constants $\alpha_{j} \ge 0$, for $j = 0, \dotsc, n$; one requires $\alpha_{0}, \alpha_n > 0$ and calls $n$ the \emph{order} of the metric. Here, $D_{s}=\frac{\partial_\theta}{\abs{c^{\prime}}}$ denotes \emph{differentiation with respect to arc length}, $\ud s=\abs{c'}\ud \theta$ \emph{integration with respect to arc length} and $\ell_{c}= \int_{\Circle}\ud s$ the corresponding \emph{curve length}.

  Using integration by parts we obtain a formula for the operator $A_{c}$:
  \begin{equation}\label{eq:integer-order-operator}
    A^{n}_{c} = \sum_{j=0}^{n}(-1)^{j} \alpha_{j} \, D_{s}^{2j}
  \end{equation}
  for metrics with constant coefficients and
  \begin{equation}\label{eq:scale-integer-order-operator}
    \widetilde A^{n}_{c} = \sum_{j=0}^{n}(-1)^{j} \alpha_{j} \, \ell_{c}^{2j-3} \, D_{s}^{2j}\,.
  \end{equation}
  for scale-invariant metrics.

  Because it will be important later, we note that if $c$ is a constant speed curve, then the operator for a metric with constant coefficients is
  \begin{equation*}
    A^{n}_{c} = \sum_{j=0}^{n} (-1)^{j} \alpha_{j} \left(\frac{2\pi}{\ell_{c}}\right)^{2j} \partial_{\theta}^{2j}\,;
  \end{equation*}
  it is a differential operator with constant coefficients and the coefficients depend only on the length of the curve.
\end{expl}

Until now local and global well-posedness results have been established only for the Sobolev metrics of integer order \cite{MM2007,Bru2015,BMM2014,BHM2011}. The main goal of this article is to extend these results to \emph{metrics of fractional order}; in particular to metrics, for which the operator $A_{c}$ is defined using Fourier multipliers of a certain class.

Given a curve $c \in \Imm(\Circle,\RR^d)$, let $\psi_{c} \in \Diff(\Circle)$ be a diffeomorphism such that $c \circ \psi_{c}^{-1}$ has constant speed. Reparametrization invariance of the metric implies that
\begin{equation}\label{eq:invariant-metric}
  G_{c}(h,k) = G_{c\circ\psi_{c}^{-1}} (h\circ\psi_{c}^{-1},k\circ\psi_{c}^{-1})\,;
\end{equation}
in other words, $G$ is determined by its behaviour on constant speed curves.

\begin{rem}
  The situation is similar to that of right-invariant metrics on diffeomorphism groups, which are determined by their behaviour at the identity diffeomorphism. For curves, the invariance property is weaker and the space of constant speed curves is still quite large.
\end{rem}

Let us write the invariance property~\eqref{eq:invariant-metric} in terms of $A_{c}$: by a straight-forward calculation we obtain
\begin{align*}
  G_{c}(h,k) & = G_{c\circ\psi_{c}^{-1}} (h\circ\psi_{c}^{-1},k\circ\psi_{c}^{-1})
  \\
             & = \int_{\Circle} \sprod{A_{c \circ \psi_{c}^{-1}} (h\circ\psi_{c}^{-1})}{k\circ\psi_{c}^{-1}} \frac{\ell_{c}}{2\pi} \ud \theta
  \\
             & = \int_{\Circle} \sprod{ R_{\psi_{c}}\circ A_{c \circ \psi_{c}^{-1}} \circ R_{\psi_{c}^{-1}}(h)}{k} \ud s\,,
\end{align*}
which implies the identity
\begin{equation*}
  A_{c} = R_{\psi_{c}}\circ A_{c \circ \psi_{c}^{-1}} \circ R_{\psi_{c}^{-1}}\,.
\end{equation*}

The class of metrics on constant speed curves is large. To make it more manageable we will restrict the possible dependance of $A_{c}$ on the curve $c$.

\begin{ass*}
  We assume from this point onwards, that the operator $A_{c \circ \psi_{c}^{-1}}$ depends on $c$ only through its length $\ell_{c}$; in other words
  \begin{equation*}
    A_{c \circ \psi_{c}^{-1}} = A(\ell_{c}),
  \end{equation*}
  where $\lambda \mapsto A(\lambda)$ is a smooth curve with values in the space of linear maps, $L(C^\infty(\Circle,\RR^d), C^\infty(\Circle,\RR^d))$\footnote{This is equivalent to requiring that the map $\RR_+ \times C^\infty(\Circle,\RR^d) \to C^\infty(\Circle,\RR^d)$ given by $(\lambda, h) \mapsto A(\lambda) h$ is smooth.}.
\end{ass*}

Then, with this assumption,
\begin{equation}\label{eq:Ac-defines-Al}
  A_{c} = R_{\psi_{c}}\circ A(\ell_{c}) \circ R_{\psi_{c}^{-1}}\,.
\end{equation}

Requiring this form for the operator $A_{c}$ imposes a restriction on $A(\ell_{c})$. For $\alpha \in \Circle$ consider $\varphi_\alpha \in \Diff(\Circle)$, defined as $\varphi_\alpha(\theta) = \theta + \alpha \mod 2\pi$. If a curve $c$ has constant speed, then so does $c \circ \varphi_\alpha$ and thus $\ell_{c} = \ell_{c \circ \varphi_\alpha}$. Therefore~\eqref{eq:Ac-defines-Al} implies
\begin{equation*}
  A(\ell_{c}) \circ R_{\varphi_\alpha} = R_{\varphi_\alpha} \circ A(\ell_{c})\,,
\end{equation*}
or equivalently after differentiating with respect to $\alpha$,
\begin{equation*}
  A(\ell_{c}) \circ \partial_\theta = \partial_\theta \circ A(\ell_{c})\,.
\end{equation*}
This means that $A(\ell_{c})$ has to be a \emph{Fourier multiplier}, i.e.,
\begin{equation*}
  A(\ell_{c}).u(\theta) = \sum_{m \in \ZZ} \mathbf{a}(\ell_{c}, m).\hat{u}(m) \exp(i m \theta)\,,
\end{equation*}
where $\mathbf{a}(\ell_{c}, \cdot) : \ZZ \to \mathcal{L}(\CC^{d})$ is called the \emph{symbol} of $A(\ell_{c})$. We will write $A(\ell_{c}) = \mathbf{a}(\ell_{c}, D)$ or $A(\ell_{c}) = \op{\mathbf{a}(\ell_{c}, \cdot)}$. For now we make no assumptions on the symbol apart from requiring that the map $\ell_{c} \mapsto A(\ell_{c})$ is smooth. More control on the symbol will be necessary in order to prove well-posedness of the geodesic equation.

\begin{expl}[Integer order Sobolev metrics]
  For Sobolev metrics with constant coefficients the symbol of the operator $A^{n}(\ell_{c})$ is
  \begin{equation*}
    \mathbf{a}^{n}(\ell_{c}, m) = \left(
    \sum_{j=0}^{n} (-1)^{j} \alpha_{j} \left(\frac{2\pi m}{\ell_{c}}\right )^{2j} \right)
    \mathbf I_d\,,
  \end{equation*}
  with $\mathbf I_d \in \CC^{d\times d}$ the identity matrix, and for scale-invariant metrics it is
  \begin{equation*}
    \mathbf{\widetilde a}^{n}(\ell_{c}, m) = \left(
    \ell_{c}^{-3} \sum_{j=0}^{n} (-1)^{j} \alpha_{j} \left(2\pi m\right )^{2j} \right)
    \mathbf I_d\,,
  \end{equation*}
\end{expl}

\begin{rem}
  Even though all known situations correspond to Fourier multipliers of the form $A(\ell_{c}) = \mathbf{a}(\ell_{c}, D)$ with $\mathbf{a}(\ell_{c}, m) = a(\ell_{c}, m) \mathbf I_d$, a multiple of the identity matrix, we will treat in this article general matrix-valued symbols, because doing so introduces no additional difficulties.
\end{rem}

\section{The geodesic equation}
\label{sec:geodesic-equations}

Geodesics between two curves $c_{0}, c_{1} \in \Imm(\Circle,\RR^{d})$ are critical points of the \emph{energy functional}
\begin{align*}
  E(c) & = \frac{1}{2} \int_{0}^{1}\int_{\Circle}\langle A_{c}c_{t}, c_{t} \rangle \ud s \ud t\,,
\end{align*}
where $c=c(t, \theta)$ is a path in $\Imm(\Circle,\RR^{d})$ joining $c_{0}$ and $c_{1}$. The geodesic equation is obtained by computing the derivative of $E$ on paths with fixed endpoints. We will use the following notations
\begin{align*}
  c_{t}      & := \partial_{t}c      &
  c^{\prime} & := \partial_{\theta}c &
  v & := \frac{c^{\prime}}{\abs{c^{\prime}}} = D_{s} c\,,
\end{align*}
and
\begin{align*}
  A_{c} & = R_{\psi_{c}} \circ A(\ell_{c}) \circ R_{\psi_{c}}^{-1} &
  A_{c}^{\prime} &:= R_{\psi_{c}} \circ A^{\prime}(\ell_{c}) \circ R_{\psi_{c}}^{-1}\,,
\end{align*}
where $A^{\prime}(\ell_{c})$ is the derivative of $A(\ell_{c})$ with respect to the parameter $\ell_{c}$ and
\begin{equation*}
  \psi_{c}(\theta) = \frac{2\pi}{\ell_{c}} \int_{0}^\theta \abs{c'} \ud \sigma\,.
\end{equation*}
We have the following commutation rules for $D_{s}$ and $R_{\psi_{c}}$,
\begin{align*}
  D_{s} \circ R_{\psi_{c}} & = \frac{2\pi}{\ell_{c}} R_{\psi_{c}} \circ \partial_\theta &
  \partial_\theta \circ R_{\psi_{c}^{-1}} &= \frac{\ell_{c}}{2\pi} R_{\psi_{c}^{-1}} \circ D_{s}\,,
\end{align*}
which imply that because $A(\ell_{c})$ commutes with $\partial_\theta$, that the operator $A_{c}$ commutes with $D_{s}$,
\begin{equation*}
  D_{s} \circ A_{c} = A_{c} \circ D_{s}\,.
\end{equation*}
This will be used when computing the geodesic equation.

\begin{thm}\label{thm:geodesic-equations}
  Assume that, for each $\lambda \in \RR^+$, the operator
  \begin{equation*}
    A(\lambda) : C^\infty(\Circle,\RR^d) \to C^\infty(\Circle,\RR^d)
  \end{equation*}
  is invertible with a continuous inverse. Then the weak Riemannian metric~\eqref{eq:metric} on $\operatorname{Imm}(\Circle,\RR^{d})$ has a geodesic spray which is given by
  \begin{equation}\label{eq:spray-imm}
    F: (c,h) \mapsto \left(h, S_{c}(h)\right), \qquad T\Imm(\Circle,\RR^d) \to TT\Imm(\Circle,\RR^d)
  \end{equation}
  where
  \begin{multline}\label{eq:geodesic-spray-imm}
    S_{c}(h) = -A_{c}^{-1}\left\{ (D_{c,h} A_{c}) h + \langle D_{s} h, v \rangle A_{c} h \right.
    \\
    + \left. \langle A_{c} h, D_{s} h \rangle v + (w + w_{0}) D_{s} v \right\}\,.
  \end{multline}
  with
  \begin{equation*}
    w(c,h) = \int_{\Circle} \langle A_{c} h, D_{s} h\rangle \ud s
  \end{equation*}
  and
  \begin{equation*}
    w_{0}(c,h)  = \int_{\Circle} \frac 1{2\pi} \langle A_{c} h, \psi_{c} D_{s} h \rangle + \frac 12 \langle (\ell_{c}^{-1} A_{c} + A_{c}') h, h \rangle \ud s\,.
  \end{equation*}
  The geodesic equation is
  \begin{equation}\label{eq:geodesic-equations}
    (A_{c} c_{t})_{t} = - \langle D_{s} c_{t}, v \rangle A_{c} c_{t}
    - \langle A_{c} c_{t}, D_{s} c_{t} \rangle v - (w(c,c_{t}) + w_{0}(c,c_{t})) D_{s} v\,.
  \end{equation}
\end{thm}

To compute these equations, we will first derive a variational formula for the reparametrization function $\psi_{c}$.

\begin{lem}\label{variation:psi}
  The derivative of the map $c \mapsto \psi_{c}$ is given by
  \begin{equation*}
    D_{c,h}\psi_{c}(\theta) = \frac{2\pi}{\ell_{c}} \int_{0}^{\theta}
    \sprod{D_{s} h}{v} \ud \tilde s - \frac{1}{\ell_{c}} \left(\int_{\Circle} \sprod{D_{s} h}{v} \ud s \right) \psi_{c}(\theta) \,.
  \end{equation*}
\end{lem}

\begin{proof}
  Using the variational formulas for $\abs{c'}$ and $\ell_{c}$,
  \begin{align*}
    D_{c,h} \abs{c'} & = \langle D_{s} h, v \rangle \abs{c'} &
    D_{c,h} \ell_{c} &= \int_{\Circle} \sprod{D_{s} h}{v} \ud s \,,
  \end{align*}
  the lemma follows directly from
  \begin{equation*}
    \psi_{c}(\theta) = \frac{2\pi}{\ell_{c}} \int_{0}^{\theta} \abs{c^{\prime}} \ud \sigma\,.\qedhere
  \end{equation*}
\end{proof}

\begin{proof}[Proof of Theorem~\ref{thm:geodesic-equations}]
  We can write the energy of a path $c(t,\theta)$ as
  \begin{align*}
    E(c) & = \frac{1}{2} \int_{0}^{1} \frac{\ell_{c}}{2\pi} \int_{\Circle} \sprod{A(\ell_{c})(c_{t}\circ\psi_{c}^{-1})}{c_{t}\circ\psi_{c}^{-1}} \ud \theta \, \ud t\,.
  \end{align*}
  The variation of the energy in direction $h$ is then given by
  \begin{align*}
    D_{c,h}E & = \int_{0}^1 \frac{\ell_{c}}{2\pi} \int_{\Circle}
    \sprod{A(\ell_{c}) (c_{t} \circ \psi_{c}^{-1})}
    {h_{t} \circ \psi_{c}^{-1} + \left(c^{\prime}_{t} \circ \psi_{c}^{-1} \right) D_{c,h}\,\psi_{c}^{-1} } \ud \theta \\
             & \qquad{} + \frac 12 \frac {\ell_{c}}{2\pi} D_{c,h}\ell_{c} \int_{\Circle} \sprod{\left( \ell_{c}^{-1}A(\ell_{c}) + A'(\ell_{c})\right)(c_{t} \circ \psi_{c} ^{-1})}{c_{t}\circ\psi_{c}^{-1}} \ud \theta \ud t \\
             & = \int_{0}^1 \int_{\Circle} \sprod{ A_{c} c_{t}}{h_{t} + c^{\prime}_{t} \left( D_{c,h} \psi_{c}^{-1}\right) \circ \psi_{c}} \ud s                                                                             \\
             & \qquad{} + \frac 12 D_{c,h} \ell_{c} \int_{\Circle}
    \sprod{(\ell_{c}^{-1} A_{c} + A_{c}') c_{t}}{c_{t}} \ud s \ud t \,.
  \end{align*}
  Let us start with the term involving $h_{t}$. After integrating by parts we get
  \begin{align*}
    \int_{0}^1 \int_{\Circle} \langle A_{c} c_{t}, h_{t}\rangle \ud s \ud t
      & = -\int_{0}^1 \int_{\Circle} \left\langle (A_{c} c_{t})_{t}+(A_{c} c_{t})\frac{\partial_{t} (|c^{\prime}|)}{\abs{c^{\prime}}}, h \right\rangle\! \ud s \ud t \\
      & = -\int_{0}^1 \int_{\Circle} \langle(A_{c} c_{t})_{t} + \langle D_{s} c_{t}, v \rangle A_{c} c_{t}, h \rangle \ud s \ud t\,.
  \end{align*}
  Next we deal with $\psi_{c}^{-1}$. Using the formula
  \begin{equation*}
    D_{c,h} \,\psi_{c}^{-1} = -\left(\frac{1}{\psi_{c}'} D_{c,h} \,\psi_{c}\right) \circ \psi_{c}^{-1}\,,
  \end{equation*}
  and Lemma~\ref{variation:psi} we obtain a formula for the variation of $\psi_{c}^{-1}$,
  \begin{equation*}
    \left(D_{c,h} \,\psi_{c}^{-1} \right) \circ \psi_{c}(\theta)
    = -\frac{1}{\abs{c^{\prime}}} \left( \int_{0}^{\theta} \langle D_{s} h, v \rangle \ud \tilde s
    - \frac{1}{2\pi} \left(\int_{\Circle} \langle D_{s} h, v\rangle \ud s \right) \psi_{c}(\theta) \right) \,.
  \end{equation*}
  Therefore we have
  \begin{align*}
    \int_{\Circle} & \sprod{A_{c} c_{t}}{c^{\prime}_{t} \left( D_{c,h} \psi_{c}^{-1}\right) \circ \psi_{c}} \ud s \\
               & = -\int_{\Circle} \langle A_{c} c_{t}, D_{s}c_{t} \rangle
    \left( \int_{0}^{\theta} \langle D_{s} h, v \rangle \ud \tilde s
    - \frac{1}{2\pi} \left(\int_{\Circle} \langle D_{s} h, v\rangle \ud \tilde s \right) \psi_{c} \right) \ud s \\
               & = - \int_{\Circle} \langle A_{c} c_{t}, D_{s} c_{t}\rangle
    \int_{0}^{\theta} \langle D_{s} h, v\rangle \ud \tilde s \ud s \\
               & \qquad{} + \frac 1{2\pi} \int_{\Circle} \langle D_{s} h, v \rangle \ud s
    \int_{\Circle} \langle A_{c} c_{t}, \psi_{c} D_{s} c_{t}\rangle \ud s
  \end{align*}
  Because $A_{c}$ is symmetric and commutes with $D_{s}$,
  \begin{equation*}
    \int_{\Circle} \langle A_{c} c_{t}, D_{s} c_{t}\rangle \ud s = 0\,,
  \end{equation*}
  and thus the function
  \begin{equation*}
    w(\theta) =  \int_{0}^{\theta} \langle A_{c} c_{t}, D_{s} c_{t} \rangle \ud s \,,
  \end{equation*}
  is periodic with $D_{s} w = \langle A_{c} c_{t}, D_{s} c_{t}\rangle$. Next we integrate by parts,
  \begin{align*}
    \int_{\Circle} & \sprod{A_{c} c_{t}}{c^{\prime}_{t} \left( D_{c,h} \psi_{c}^{-1}\right) \circ \psi_{c}} \ud s \\
               & = \int_{\Circle} w \langle D_{s} h, v\rangle \ud s
    + \frac 1{2\pi} \int_{\Circle} \langle D_{s} h, v \rangle \ud s
    \int_{\Circle} \langle A_{c} c_{t}, \psi_{c} D_{s} c_{t}\rangle \ud s \\
               & = -\int_{\Circle} \langle D_{s} (wv), h \rangle \ud s
    - \frac 1{2\pi} \int_{\Circle} \langle D_{s} v, h \rangle \ud s
    \int_{\Circle} \langle A_{c} c_{t}, \psi_{c} D_{s} c_{t}\rangle \ud s\,.
  \end{align*}
  Finally we consider the term involving $D_{c,h} \ell_{c}$.
  Since
  \begin{equation*}
    D_{c,h}\ell_{c}
    = \int_{\Circle} \langle D_{s} h, v \rangle \ud s
    = - \int_{\Circle} \langle D_{s} v, h \rangle \ud s\,,
  \end{equation*}
  we have
  \begin{multline*}
    \frac 12 D_{c,h} \ell_{c} \int_{\Circle}
    \langle (\ell_{c}^{-1} A_{c} + A_{c}') c_{t}, c_{t} \rangle \ud s = \\
    = - \frac 12 \int_{\Circle} \langle D_{s} v, h \rangle \ud s \int_{\Circle} \langle (\ell_{c}^{-1} A_{c} + A_{c}') c_{t}, c_{t} \rangle \ud s\,.
  \end{multline*}
  Grouping all these expressions together, we can express $D_{c,h}E$ as
  \begin{multline}\label{eq:DE-formula}
    D_{c,h} E(c) = \\
    \int_{0}^1 \int_{\Circle}
    \left\langle -(A_{c} c_{t})_{t} - \langle D_{s} c_{t}, v \rangle A_{c} c_{t}
    - D_{s}(wv) - w_{0} D_{s} v, h \right\rangle \ud s \ud t\,,
  \end{multline}
  where
  \begin{equation*}
    w_{0} = \int_{\Circle} \left( \frac{1}{2\pi} \langle A_{c} c_{t}, \psi_{c} D_{s} c_{t} \rangle
    + \frac{1}{2} \langle (\ell_{c}^{-1} A_{c} + A_{c}') c_{t}, c_{t} \rangle \right) \ud s\,.
  \end{equation*}
  Thus we obtain the geodesic equation
  \begin{equation*}
    (A_{c} c_{t})_{t} = - \langle D_{s} c_{t}, v \rangle A_{c} c_{t}
    - D_{s}(wv) - w_{0} D_{s} v\,.
  \end{equation*}
  The existence of the geodesic spray follows from~\eqref{eq:DE-formula}, which can be rewritten as
  \begin{equation*}
    D_{c,h}E(c) = \int_{0}^1 G_{c}(-c_{tt} + S_{c}(c_{t}), h) \ud t\,.
  \end{equation*}
  Note that in this last step we used the invertibility of $A_{c}$ on $C^\infty(\Circle,\RR^d)$.
\end{proof}

Next we will see how the geodesic equation simplifies for Sobolev metrics with constant coefficients and scale-invariant metrics. First note, that using integration by parts we can write
\begin{align*}
  \frac 1{2\pi} \int_{\Circle} \langle A_{c} c_{t}, \psi_{c} D_{s} c_{t} \rangle \ud s
    & = \frac 1{\ell_{c}} \int_{\Circle} \langle A_{c} c_{t}, D_{s} c_{t} \rangle
  \int_{0}^\theta \abs{c'} \ud \sigma \ud s \\
    & = - \frac 1{\ell_{c}} \int_{\Circle} \int_{0}^\theta \langle A_{c} c_{t}, D_{s} c_{t} \rangle
  \ud \tilde s \ud s\,.
\end{align*}
We shall also make use of the following identity, valid for $j \geq 1$ and $h \in C^\infty(\Circle,\RR^d$),
\begin{equation*}
  \langle D_{s}^{2j}h, D_{s}h\rangle = D_{s} \left(\frac{1}{2} \sum_{k=1}^{2j-1} (-1)^{k+1} \langle D_{s}^{2j-k}h, D_{s}^{k}h \rangle \right)\,.
\end{equation*}
Setting
\begin{align*}
  W_{j} & = \frac{1}{2} \sum_{k=1}^{2j-1} (-1)^{k+1} \langle D_{s}^{2j-k}c_{t}, D_{s}^{k}c_{t} \rangle \;\;\text{for }j\geq 1\,,
        & W_{0}
        & = \frac 12 \abs{c_{t}}^{2} \,,
\end{align*}
we obtain
\begin{equation*}
  \int_{0}^\theta \langle D_{s}^{2j}c_{t}, D_{s}c_{t}\rangle \ud \tilde s = W_{j}(\theta) - W_{j}(0)\,,
\end{equation*}
as well as
\begin{equation*}
  \int_{\Circle} W_{j}(\theta) \ud s = \frac 12 (1-2j) \int_{\Circle} \langle D_{s}^{2j} c_{t}, c_{t} \rangle \ud s\,.
\end{equation*}

\begin{expl}[Metrics with constant coefficients]
  \label{ex:metrics-constant-coeff}
  For metrics with constant coefficients we have
  \begin{align*}
    A^{n}_{c} & = \sum_{j=0}^{n} (-1)^{j} \alpha_{j} D_{s}^{2j},
              & A^{n}(\ell_{c})
              & = \sum_{j=0}^{n} (-1)^{j} \alpha_{j} \left( \frac{2\pi}{\ell_{c}} \right)^{2j} \partial_\theta^{2j}\,,
  \end{align*}
  and thus their $\ell_{c}$-derivatives are
  \begin{align*}
    (A^{n}_{c})' & = -2\ell_{c}^{-1} \sum_{j=0}^{n} (-1)^{j} j \alpha_{j} D_{s}^{2j},
                 & A'(\ell_{c})
                 & = -2 \ell_{c}^{-1} \sum_{j=0}^{n} (-1)^{j} j \alpha_{j} \left( \frac{2\pi}{\ell_{c}} \right)^{2j} \partial_\theta^{2j}\,.
  \end{align*}
  Therefore
  \begin{equation*}
    \ell_{c}^{-1} A_{c}^{n} + (A_{c}^{n})' = \ell_{c}^{-1} \sum_{j=0}^{n} (-1)^{j} (1-2j) \alpha_{j} D_{s}^{2j}\,,
  \end{equation*}
  and
  \begin{align*}
    w_{0} & = \frac 1{\ell_{c}} \sum_{j=0}^{n} (-1)^{j} \alpha_{j} \left[- \int_{\Circle} W_{j}(\theta) - W_{j}(0) \ud s + \frac 1{2}(1-2j) \int_{\Circle} \langle D_{s}^{2j} c_{t}, c_{t} \rangle \ud s \right] \\
          & = \sum_{j=0}^{n} (-1)^{j} \alpha_{j} W_{j}(0)\,.
  \end{align*}
  Since
  \begin{equation*}
    w = \sum_{j=0}^{n} (-1)^{j} \alpha_{j} \int_{0}^\theta \langle D_{s}^{2j} c_{t}, D_{s} c_{t} \rangle \ud \tilde s
    = \sum_{j=0}^{n} (-1)^{j} \alpha_{j} \left( W_{j}(\theta) - W_{j}(0) \right)\,,
  \end{equation*}
  it follows that
  \begin{equation*}
    w(\theta)+w_{0} = \sum_{j=0}^{n} (-1)^{j} \alpha_{j} W_{j}(\theta)\,.
  \end{equation*}
  Thus the geodesic equation has the form
  \begin{equation*}
    (A_{c} c_{t})_{t} = - \langle D_{s} c_{t}, v \rangle A_{c} c_{t}
    - \langle A_{c} c_{t}, D_{s} c_{t} \rangle v - \left( \sum_{j=0}^{n} (-1)^{j} \alpha_{j} W_{j} \right) D_{s} v\,.
  \end{equation*}
  Thus we have regained the formula from \cite{MM2007}.
\end{expl}

\begin{expl}[Scale-invariant metrics]
  For scale-invariant metrics we have
  \begin{align*}
    \tilde A^{n}_{c} & = \sum_{j=0}^{n} (-1)^{j} \alpha_{j} \ell_{c}^{2j-3} D_{s}^{2j},
                     & \tilde A^{n}(\ell_{c})
                     & = \ell_{c}^{-3} \sum_{j=0}^{n} (-1)^{j} \alpha_{j} (2\pi)^{2j} \partial_\theta^{2j}\,,
  \end{align*}
  and thus their $\ell_{c}$-derivatives are
  \begin{align*}
    (A^{n}_{c})' & = -3\ell_{c}^{-1} \sum_{j=0}^{n} (-1)^{j} \alpha_{j} \ell_{c}^{2j-3} D_{s}^{2j},
                 & \tilde A'(\ell_{c})
                 & = -3 \ell_{c}^{-4} \sum_{j=0}^{n} (-1)^{j} \alpha_{j} \left( 2\pi \right)^{2j} \partial_\theta^{2j}\,.
  \end{align*}
  Therefore
  \begin{equation*}
    \ell_{c}^{-1} \tilde A^{n}_{c} + (\tilde A^{n}_{c})' = -2 \ell_{c}^{-1} \tilde A^{n}_{c}\,.
  \end{equation*}
  A similar calculation as before gives
  \begin{align*}
    w_{0}     & = \sum_{j=0}^{n} (-1)^{j} \alpha_{j} \ell_{c}^{2j-3} \left(W_{j}(0) + \frac{1}{\ell_{c}}\left(j + \frac 12\right) \int_{\Circle} \langle D_{s}^{2j} c_{t}, c_{t} \rangle \ud s\right) \\
    w(\theta) & = \sum_{j=0}^{n} (-1)^{j} \alpha_{j} \ell_{c}^{2j-3} \left( W_{j}(\theta) - W_{j}(0) \right)\,,
  \end{align*}
  and therefore
  \begin{align*}
    w(\theta) + w_{0} = \sum_{j=0}^{n} (-1)^{j} \alpha_{j} \ell_{c}^{2j-3} \left(W_{j}(\theta) + \frac{1}{\ell_{c}}\left(j + \frac 12\right) \int_{\Circle} \langle D_{s}^{2j} c_{t}, c_{t} \rangle \ud s\right)\,.
  \end{align*}
\end{expl}

\section{Smoothness of the metric}
\label{sec:smoothness-metric}

For any $q > \frac32$ we can consider the Sobolev completion $\mathcal{I}^{q} = \mathcal{I}^{q}(\Circle,\RR^d)$, which is an open set of the Hilbert vector space $\HRd{q}$. It is the aim of this section to show that the metric $G_{c}$ extends to a smooth weak metric on $\mathcal{I}^{q}$, for high enough $q$. It turns out that the smoothness of the metric reduces to the smoothness of the mapping
\begin{equation*}
  c \mapsto A_{c} = R_{\psi_{c}} \circ A(\ell_{c}) \circ R_{\psi_{c}^{-1}} .
\end{equation*}
We will begin our investigations with this question.

It is well-known (see for instance~\cite[Appendix A]{EM1970}) that, given a \emph{differential operator} $A$ of order $r$ with \emph{smooth coefficients} and a diffeomorphism $\psi$, the conjugate operator $A_\psi = R_\psi \circ A \circ R_{\psi^{-1}}$ is again a differential operator of order $r$, whose coefficients are polynomial expressions in $\psi$, its derivatives and $1/\psi'$; in particular, the mapping
\begin{equation*}
  \psi \mapsto A_{\psi} = R_{\psi} \circ A \circ R_{\psi^{-1}} , \quad \D{q}(\Circle) \to \mathcal{L}(\HRd{q},\HRd{q-r})
\end{equation*}
is smooth for $q > 3/2$ and $q \ge r \ge 1$, where $\D{q}(\Circle)$ is group of $H^{q}$-diffeomorphisms of $\Circle$. In~\cite{EK2014} Escher and Kolev extended this result to \emph{Fourier multipliers} by showing that the map $\psi \mapsto A_\psi$ remains smooth when $A$ is a Fourier multiplier of class $\mathcal{S}^{r}(\ZZ)$ (to be defined below).

The proof that the mapping $c\mapsto A_{c}$ is smooth that we present here is inspired by~\cite{EK2014,BEK2015}, but in the present case we have to deal with Fourier multipliers that depend (in a nice way) on a parameter $\lambda$.

First we introduce the classes $\mathcal{S}^{r}(\ZZ)$ and $\mathcal{S}^{r}_\lambda(\ZZ)$ of Fourier multipliers (for further details see~\cite{RT2010} for instance). A Fourier multiplier $\mathbf a(D)$ acts on a function $u$ via
\begin{equation}
  \mathbf a(D).u(\theta) = \sum_{m \in \ZZ} \mathbf a(m). \hat u(m) \exp(im\theta)\,,
\end{equation}
and its symbol is a function $\mathbf a : \ZZ \to \mathcal{L}(\CC^{d})$.

\begin{defn}
  Given $r \in \RR$, the Fourier multiplier $\mathbf{a}(D)$ belongs to the class $\mathcal{S}^{r}(\ZZ)$, if $\mathbf{a}(m)$ satisfies
  \begin{equation*}
    \norm{\Delta^{\alpha} \mathbf{a}(m)} \leq C_{\alpha} \langle m \rangle^{r-\alpha} ,
  \end{equation*}
  for each $\alpha \in \NN$, where $\langle m \rangle := (1+\abs{m}^{2})^{1/2}$.
\end{defn}

Here we define the difference operator $\Delta^{\alpha}$ via
\begin{equation*}
  \Delta \mathbf{a}(m) = \mathbf{a}(m+1) - \mathbf{a}(m),\qquad \Delta^{\alpha} = \Delta \circ \Delta^{\alpha-1}\,.
\end{equation*}
We equip the space $\mathcal{S}^{r}(\ZZ)$ with the topology induced by the seminorms
\begin{equation*}
  p_{\alpha}(\mathbf{a}(D)) = \sup_{m \in \ZZ} \norm{\Delta^{\alpha} \mathbf{a}(m)} \langle m \rangle^{-(r-\alpha)}\,,
\end{equation*}
where $\|\cdot\|$ is some norm on $\mathcal L(\CC^d)$. With this topology $\mathcal{S}^{r}(\ZZ)$ is a Fr\'{e}chet space. The class $\mathcal S^r(\ZZ)$ coincides with the one defined in~\cite{EK2014}; the equivalence is shown in~\cite{RT2010}.

\begin{expl}
  Any linear differential operator of order $r$ with constant coefficients belongs to $\mathcal{S}^{r}(\ZZ)$. Furthermore the operator $\Lambda^{2r} = \op{\langle m \rangle^{2r}}$, which defines the $H^{r}$-norm via
  \begin{equation*}
    \norm{u}_{H^r} = \int_{\Circle} \Lambda^{2r} u \cdot u \ud \theta,
  \end{equation*}
  belongs to $\mathcal{S}^{2r}(\ZZ)$.
\end{expl}

\begin{rem}
  A Fourier multiplier $\mathbf{a}(D)$ of class $\mathcal{S}^{r}(\ZZ)$ extends for any $q \in \RR$ to a \emph{bounded linear operator}
  \begin{equation*}
    \mathbf a(D): H^{q}(\Circle,\RR^{d}) \to H^{q-r}(\Circle,\RR^{d})\,,
  \end{equation*}
  and the linear embedding
  \begin{equation*}
    \mathcal{S}^{r}(\ZZ) \to \mathcal{L}(H^{q}(\Circle,\RR^d), H^{q-r}(\Circle,\RR^d))
  \end{equation*}
  is continuous.
\end{rem}

Now we introduce the class of parameter-dependent symbols.

\begin{defn}
  Given $r \in \RR$, the class $\mathcal{S}^{r}_{\lambda}(\ZZ) = C^\infty(\RR^+, \mathcal{S}^{r}(\ZZ))$ consists of smooth curves $\RR^+ \to \mathcal{S}^{r}(\ZZ)$, $\lambda \mapsto \mathbf{a}(\lambda, D)$ of Fourier multipliers. We will call such a curve a \emph{$\lambda$-symbol}.
\end{defn}

We will often write $\mathbf a_\lambda(D)$ for $\mathbf a(\lambda, D)$ or, by slight abuse of notation, to denote the whole curve $\lambda \mapsto \mathbf a(\lambda, D)$. Using the material in Appendix~\ref{sec:smooth-curves}, we can give an alternative description of the space $\mathcal{S}^{r}_{\lambda}(\ZZ)$ of one-parameter families of symbols.

\begin{lem}
  Given $r \in \RR$, a family of smooth curves $\mathbf{a}(\cdot, m) \in C^\infty(\RR^+, \mathcal{L}(\CC^d))$ with $m \in \ZZ$ defines an element $\mathbf{a}_\lambda(D)$ in the class $\mathcal{S}^{r}_{\lambda}(\ZZ)$ if and only if for each $\alpha, \beta \in \NN$,
  \begin{equation}\label{eq:symbol-la-est}
    \norm{\partial^\beta_{\lambda} \Delta^{\alpha} \mathbf{a}(\lambda, m)} \leq C_{\alpha, \beta} \langle m\rangle ^{r-\alpha}\,,
  \end{equation}
  holds locally uniformly in $\lambda\in\RR^+$.
\end{lem}

\begin{proof}
  Let $\mathbf{a}_{\lambda}(D) \in \mathcal{S}^{r}_{\lambda}(\ZZ)$. Since all derivatives of a smooth curve are locally bounded, it follows that for all $\alpha, \beta \in \NN$, the expression $p_{\alpha}\left(\partial_{\lambda}^\beta \mathbf{a}_{\lambda}(D)\right)$ is bounded locally uniformly in $\lambda$ and hence $\mathbf{a}(\lambda, m)$ satisfies~\eqref{eq:symbol-la-est} for some constants $C_{\alpha, \beta}$.

  Conversely, let a family $\mathbf{a}(\lambda, m)$ of smooth curves satisfying the above estimates be given. Define the curves $\mathbf{a}^\beta(\lambda, m) = \partial_{\lambda}^\beta \mathbf{a}(\lambda, m)$. The estimate~\eqref{eq:symbol-la-est} shows that the curves $\lambda \mapsto a^\beta(\lambda, \cdot)$ are locally bounded in $\mathcal{S}^{r}(\ZZ)$ and hence by Lemma~\ref{lem:smooth-curves} the curve $\lambda \mapsto \mathbf{a}_{\lambda}(D)$ is an element of $\mathcal{S}^{r}_{\lambda}(\ZZ)$.
\end{proof}

The following is the main theorem that will be used to show the smoothness of the metric. Set
\begin{equation*}
  A_{c} = R_{\psi_{c}} \circ A(\ell_{c}) \circ R_{\psi_{c}^{-1}}\,,
\end{equation*}
with $\psi_{c}(\theta) = \frac{2\pi}{\ell_{c}}\int_{0}^\theta \abs{c'} \ud \sigma$.

\begin{prop}\label{prop:smoothness-Ac}
  Let $r \geq 1$ and $A(\lambda) = \mathbf{a}_\lambda(D)$ belong to the class $\mathcal{S}^{r}_{\lambda}(\ZZ)$.
  Then the map
  \begin{align*}
    c & \mapsto A_{c}\,, & \mathcal{I}^{q}(\Circle,\RR^d) & \to \mathcal{L}(H^{q}(\Circle,\RR^d), H^{q-r}(\Circle,\RR^d))\,,
  \end{align*}
  is smooth provided $q > \frac 32$ and $q \geq r$.
\end{prop}

\begin{proof}
  It was established in~\cite[Thm.~3.7]{EK2014} that for $A \in \mathcal{S}^{r}(\ZZ)$ the mapping
  \begin{equation*}
    \psi \mapsto A_\psi := R_{\psi} \circ A \circ R_{\psi^{-1}} \,, \quad \D{q} \to \mathcal{L}(H^q, H^{q-r})
  \end{equation*}
  is smooth. Using the uniform boundedness principle~\cite[5.18]{KM1997} it follows that the map
  \begin{equation*}
    \psi \mapsto (A \mapsto A_\psi)\,, \quad \D{q} \to \mathcal L(\mathcal S^r(\ZZ), \mathcal{L}(H^q, H^{q-r}))
  \end{equation*}
  is smooth\footnote{When $E,F$ are Fr\'echet spaces or more generally convenient vector spaces, $\mathcal L(E,F)$ is the space of bounded linear maps equipped with the topology of uniform convergence on bounded sets; see~\cite[5.3]{KM1997} for details.}. Because the inclusion
  \begin{equation*}
    \mathcal L(\mathcal S^r(\ZZ), \mathcal L(H^q,H^{q-r})) \subset C^\infty(\mathcal S^r(\ZZ), \mathcal L(H^q,H^{q-r})\,,
  \end{equation*}
  is bounded~\cite[5.3]{KM1997}, the following map is smooth
  \begin{equation*}
    \psi \mapsto (A \mapsto A_\psi)\,, \quad \D{q} \to C^\infty(\mathcal S^r(\ZZ), \mathcal{L}(H^q, H^{q-r}))\,;
  \end{equation*}
  via the exponential law~\cite[3.12]{KM1997} this is equivalent to the smoothness of the joint map
  \begin{equation*}
    (A, \psi) \mapsto A_\psi\,, \quad \mathcal S^r(\ZZ) \times \D{q} \to \mathcal{L}(H^q, H^{q-r}))\,.
  \end{equation*}
  Next we note that the maps
  \begin{equation*}
    c \mapsto \ell_{c}\,, \quad \mathcal{I}^{q} \to \RR^+ \qquad\text{and}\qquad
    c \mapsto \psi_{c}\,, \quad \mathcal{I}^{q} \to \D{q}
  \end{equation*}
  are smooth -- this can be seen from their definitions -- and since $\mathcal S^r_\lambda(\ZZ) = C^\infty(\RR^+, \mathcal S^r(\ZZ))$ the composition
  \begin{equation*}
    c \mapsto A(\ell_{c})\,, \quad \mathcal I^q \to \mathcal S^r(\ZZ)
  \end{equation*}
  is smooth as well. To conclude the proof we note that $c \mapsto A_{c}$ can be written as the composition $c \mapsto (A(\ell_{c}), \psi_{c}) \mapsto A(\ell_{c})_{\psi_{c}}$.
\end{proof}

\begin{cor}\label{cor:smooth-G-extension}
  Let $r \geq \frac 12$ and $A(\lambda) = \mathbf{a}_\lambda(D)$ belong to the class $\mathcal{S}^{2r}_{\lambda}(\ZZ)$.
  Then the bilinear form
  \begin{equation*}
    G_{c}(h,k) = \int_{\Circle} \left\langle A_{c} h, k \right\rangle \!\ud s
  \end{equation*}
  extends smoothly to $\mathcal{I}^{q}(\Circle,\RR^d)$ provided $q > \frac 32$ and $q \geq 2r$.
\end{cor}

\begin{proof}
  It is enough to show that
  \begin{equation*}
    c \mapsto  \check{G}_{c} = M_{\abs{c^{\prime}}}\circ A_{c} , \quad \mathcal{I}^{q} \to \mathcal{L}(H^q, H^{q-r}),
  \end{equation*}
  is smooth, where $M_{\abs{c^{\prime}}}$ denotes pointwise multiplication by $\abs{c^{\prime}}$. Now
  \begin{equation*}
    c \mapsto M_{\abs{c^{\prime}}} , \quad \mathcal{I}^{q} \to \mathcal{L}(H^\rho, H^\rho),
  \end{equation*}
  is smooth for $0 \leq \rho \leq q-1$ and the conclusion follows by Proposition~\ref{prop:smoothness-Ac} and the fact that composition of continuous linear mappings between Banach spaces is smooth.
\end{proof}

\section{Smoothness of the spray}
\label{sec:smoothness-spray}

In order to prove the existence and smoothness of the spray, we will require moreover an \emph{ellipticity condition} on $\lambda$-symbols. For the purpose of this article, we will adopt the following definition.

\begin{defn}
  An element $\mathbf{a}_\lambda(D)$ of the class $\mathcal{S}^{r}_{\lambda}(\ZZ)$ is called \emph{locally uniformly elliptic}, if $\mathbf{a}(\lambda, m) \in GL(\CC^d)$ for all $(\lambda, m) \in \RR^+ \times \ZZ$ and
  \begin{equation*}
    \norm{\mathbf{a}(\lambda, m)^{-1}} \leq C \langle m \rangle^{-r}
  \end{equation*}
  holds locally uniformly in $\lambda$.
\end{defn}

\begin{rem}
  A Fourier multiplier $\mathbf a(D)$ of class $\mathcal{S}^{r}(\ZZ)$ is \emph{elliptic}, if $\mathbf a(m) \in GL(\CC^d)$ for all $m \in \ZZ$ and
  \begin{equation*}
    \norm{\mathbf{a}(m)^{-1}} \leq C \langle m \rangle^{-r}\,,
  \end{equation*}
  holds for all $m$. Such an $\mathbf a(D)$ induces a \emph{bounded isomorphism} between $H^{q}(\Circle,\RR^{d})$ and  $H^{q-r}(\Circle,\RR^{d})$, for all $q\in \RR$.
\end{rem}

We summarize our considerations by introducing the following class of operators which will be denoted $\mathcal{E}^{r}_{\lambda}(\ZZ)$.

\begin{defn}
  A family $\mathbf{a}_{\lambda}(D)$ of Fourier multipliers is an element of the class $\mathcal{E}^{r}_{\lambda}(\ZZ)$, if
  \begin{enumerate}
    \item
          $\mathbf{a}_{\lambda}(D)$ is in the class $\mathcal{S}^{r}_{\lambda}(\ZZ)$,
    \item
          $\mathbf{a}(\lambda, m)$ is a positive Hermitian matrix for all $(\lambda, m) \in \RR^+ \times \ZZ$ and
    \item
          $\mathbf{a}_{\lambda}(D)$ is locally uniformly elliptic.
  \end{enumerate}
\end{defn}

\begin{thm}
  Let $r \geq 1$, $q\geq 2r$ and $A(\lambda) = \mathbf a_\lambda(D)$ belong to $\mathcal{E}^{2r}_{\lambda}(\ZZ)$. Then
  \begin{equation*}
    G_{c}(h,k) = \int_{\Circle} \left\langle A_{c} h, k \right\rangle \!\ud s
  \end{equation*}
  defines a smooth weak Riemannian metric of order $r$ on $\mathcal{I}^{q}(\Circle, \RR^d)$ with a smooth geodesic spray.
\end{thm}

\begin{proof}
  It is shown in Corollary~\ref{cor:smooth-G-extension} that $G$ extends smoothly to $\mathcal I^q(\Circle,\RR^d)$. We can write
  \begin{equation*}
    G_{c}(h,h) = 2\pi \sum_{m \in \ZZ} \langle \mathbf a(\ell_{c}, m) \hat h(m), \hat h(m) \rangle\,,
  \end{equation*}
  and because $\mathbf a(\ell_{c}, m)$ are positive Hermitian matrices, $G_{c}(h,h) = 0$ only for $h = 0$. Thus $G$ is a Riemannian metric. It is a weak Riemannian metric, because the inner product $G_{c}(\cdot,\cdot)$ induces on each tangent space $T_{c} \mathcal I^q(\Circle,\RR^d)$ the $H^r$-topology, while the manifold $\mathcal I^q(\Circle,\RR^d)$ itself carries the $H^q$-topology and by our assumptions $q > r$.

  It remains to show that the geodesic spray of $G$ exists and is smooth. For each $\lambda \in \RR^+$, the operator $A(\lambda)$ is an elliptic Fourier multiplier and thus induces a bi-bounded linear isomorphism on $C^\infty(\Circle,\RR^d)$. Thus we can apply Theorem~\ref{thm:geodesic-equations}, which shows that the geodesic spray exists on the space $\operatorname{Imm}(\Circle,\RR^d)$ of smooth immersions. We will show that the map $S_{c}(h)$ extends smoothly to the Sobolev completion $\mathcal I^q(\Circle,\RR^d)$. Then $F(c,h) = (h,S_{c}(h))$ is necessarily the geodesic spray of $G$ on $\mathcal I^q(\Circle,\RR^d)$.

  We have the following formula for $S_{c}(h)$,
  \begin{equation*}
    S_{c}(h) = -A_{c}^{-1}\left\{ (D_{c,h} A_{c}) h
    + \langle D_{s} h, v \rangle A_{c} h
    + \langle A_{c} h, D_{s} h \rangle v + (w + w_{0}) D_{s} v \right\}\,.
  \end{equation*}
  The map
  \begin{equation*}
    A \mapsto A^{-1}, \quad \mathcal U \subset \mathcal{L}(H^q, H^{q-2r}) \to \mathcal{L}(H^{q-2r}, H^q)\,,
  \end{equation*}
  defined on the open subset $\mathcal U$ of invertible operators is smooth and so is $c \mapsto A_{c}$ by Proposition~\ref{prop:smoothness-Ac}; therefore so is the composition
  $c \mapsto A_{c}^{-1} : \mathcal I^q \to \mathcal L(H^{q-2r},H^q)$. Thus we need to show that
  \begin{equation*}
    (c,h) \mapsto (D_{c,h} A_{c}) h
    + \langle D_{s} h, v \rangle A_{c} h
    + \langle A_{c} h, D_{s} h \rangle v + (w + w_{0}) D_{s} v\,,
  \end{equation*}
  is smooth between $\mathcal I^q \times H^q \to H^{q-2r}$. The first term $(D_{c,h} A_{c})h$ is the derivative of $(c,h) \mapsto A_{c}h$ with respect to the first variable and hence the map $(c,h) \mapsto (D_{c,h} A_{c})h$ is smooth between $\mathcal I^q \times H^q \to H^{q-2r}$. Arc-length derivation
  \begin{equation*}
    (c,u) \mapsto D_{s} u\,,\quad
    \mathcal I^q \times H^\rho \to H^{\rho-1}\,,\quad 0 \leq \rho \leq q\,,
  \end{equation*}
  is smooth and pointwise multiplication in $C^\infty(\Circle,\RR)$ extends to a continuous bilinear mapping
  \begin{equation*}
    H^\sigma \times H^\rho \to H^\rho\,,\quad \sigma > \frac 12 \,,\; 0 \leq \rho \leq \sigma\,.
  \end{equation*}
  Therefore, noting that $v = D_{s} c$ and $q - 1 > \frac 12$, the expression
  $\langle D_{s} h, v \rangle A_{c} h + \langle A_{c} h, D_{s} h \rangle v$ is a smooth map $\mathcal I^q \times H^q \to H^{q-2r}$. Next we use the fact that the antiderivative
  \begin{equation*}
    u \mapsto \left( \theta \mapsto \int_{0}^\theta u(\sigma) \ud \sigma \right)\,,\quad
    H^\rho \to H^{\rho+1}\,,\quad \rho \geq 0\,,
  \end{equation*}
  is a bounded linear mapping. Thus $(c,h) \mapsto w(c,h)$ maps smoothly $\mathcal I^q \times H^q \to H^{q-2r+1}$ and because $q-2r + 1 > \frac 12$ we have that $(c,h) \mapsto w(c,h) D_{s} v$ is smooth between $\mathcal I^q \times H^q \to H^{q-2r}$. The last term is $w_{0}(c,h)$. First we note that the map $c \mapsto \psi_{c}$ between $\mathcal I^q \to \mathcal D^q$ is smooth. A term-by-term inspection shows that the map $(c,h) \mapsto f(c,h)$, where $f(c,h)$ is the integrand in the definition of $w_{0}(c,h)$ is a smooth map $\mathcal I^q \times H^q \to L^1$ and thus $(c,h) \mapsto w_{0}(c,h)$ is smooth as well. Thus the extension of $G$ to $\mathcal I^q(\Circle,\RR^d)$ has a smooth spray.
\end{proof}

As a corollary, using the Cauchy--Lipschitz theorem, we obtain the local existence of geodesics on the Hilbert manifold $\mathcal{I}^{q}(\Circle,\RR^d)$.

\begin{cor}\label{cor:Hq-local-well-posedness}
  Let $A_{\lambda} = \mathbf{a}_{\lambda}(D)$ belong to the class $\mathcal{E}^{2r}_{\lambda}(\ZZ)$, where $r \geq 1$, and let $q \ge 2r$. Consider the geodesic flow on the tangent bundle $T\mathcal{I}^{q}$ induced by the Fourier multiplier $A_{\lambda}$. Then, given any $(c_{0},h_{0}) \in T\mathcal{I}^{q}$, there exists a unique non-extendable geodesic
  \begin{equation*}
    (c, h)\in C^\infty(J,T\mathcal{I}^{q})
  \end{equation*}
  with $c(0) = c_{0}$ and $h(0) = h_{0}$ on the maximal interval of existence $J$, which is open and contains $0$.
\end{cor}

As a consequence of a no-loss-no-gain argument, we also obtain local well-posedness on the space of smooth curves.

\begin{cor}\label{cor:smooth-local-well-posedness}
  Let $A_{\lambda} = \mathbf{a}_{\lambda}(D)$ belong to the class $\mathcal{E}^{2r}_{\lambda}(\ZZ)$, where $r \geq 1$. Consider the geodesic flow on the tangent bundle $T\Imm(\Circle,\RR^{d})$ induced by the Fourier multiplier $A_{\lambda}$. Then, given any $(c_{0},h_{0}) \in T\Imm(\Circle,\RR^{d})$, there exists a unique non-extendable geodesic
  \begin{equation*}
    (c, h)\in C^\infty(J,T\Imm(\Circle,\RR^{d}))
  \end{equation*}
  with $c(0) = c_{0}$ and $h(0) = h_{0}$ on the maximal interval of existence $J$, which is open and contains $0$.
\end{cor}
\begin{proof}
  Since the metric is invariant by re-parametrization, it is invariant in particular by translations of the parameter. A similar observation has been used in~\cite[Section 4]{EK2014} to show that local existence of geodesic still holds in the smooth category. The proof is based on a \emph{no loss-no gain result} in spatial regularity, initially formulated in~\cite{EM1970} (see also~\cite{Bru2016}). The same argument is still true here and leads to the local existence of the geodesics on the Fr\'{e}chet manifold $\Imm(\Circle,\RR^{d})$.
\end{proof}

\section{Strong Riemannian metrics}
\label{sec:strong-Riemannian-metricse}

The goal of this section is to show, that metrics of order $s > \frac32$ (induced by a Fourier multiplier $A_{\lambda}$ in the class $\mathcal{E}^{2s}_{\lambda}(\ZZ)$), induce \emph{strong} smooth Riemannian metrics on the Sobolev completion $\mathcal{I}^{s}$ of the same order as the metric. Let $A(\lambda)$ be of class $\mathcal E^{2s}_\lambda(\ZZ)$ with $s > \frac 32$. In Section~\ref{sec:smoothness-spray} it was shown that the corresponding Riemannian metric
\begin{equation}
  G_{c}(h,k) = \int_{\Circle} \left\langle A_{c} h, k \right\rangle \ud s\,,
\end{equation}
can be extended smoothly to $\mathcal I^q(\Circle,\RR^d)$ for $q \geq 2s$. Now we want to improve this to $q \geq s$. Of particular interest is the case $q=s$, when the topologies induced by the inner products $G_{c}(\cdot, \cdot)$ coincide with the manifold topology.

Any positive Hermitian matrix $\mathbf{a}$ has a unique positive square root $\mathbf{b}$ which depends smoothly on $\mathbf{a}$. We have, moreover, the stronger result that an operator $A(\lambda) = \mathbf{a}_{\lambda}(D)$ in the class $\mathcal{E}^{2s}_\lambda$ has a square root $B(\lambda) = \mathbf{b}_{\lambda}(D)$ in the class $\mathcal{E}^{s}_\lambda$ (see Lemma~\ref{lem:square-root}). With $B_{c} = R_{\psi_{c}} \circ B(\ell_{c}) \circ R_{\psi_{c}^{-1}}$ we have the identities
\begin{align}
  A(\ell_{c}) & = B(\ell_{c})^{2}\,, &
  A_{c} &= B_{c}^{2}\,;
\end{align}
furthermore, because $\mathbf b(\lambda, m)$ is a Hermitian matrix, the operator $B(\ell_{c})$ is $L^{2}(d\theta)$-symmetric and for each curve $c$ the operator $B_{c}$ is $L^{2}(ds)$-symmetric. Therefore we can rewrite the metric $G$ in the symmetric form
\begin{equation}\label{eq:symmetric-metric}
  G_{c}(h,k) = \int_{\Circle} \left\langle A_{c} h, k \right\rangle \ud s = \int_{\Circle} \left\langle B_{c} h, B_{c} k \right\rangle \ud s\,.
\end{equation}
We obtain therefore the following expression for the operator $\check G_{c}$ on $T\mathcal{I}^{s}$:
\begin{equation*}
  \check G_{c} = B_{c}^{t} \circ M_{\abs{c^{\prime}}}\circ B_{c} \,,
\end{equation*}
where
\begin{equation*}
  B_{c}: \HRd{s} \to \LRd, \qquad B_{c}^{t}: \LRd \to \HRd{-s}\,,
\end{equation*}
and $B_{c}^{t}$ is the transpose of $B_{c}$. The latter formula can now be used to obtain the following result concerning the smooth extension of this family of inner products into a strong Riemannian metric on $\mathcal{I}^{s}$, provided $s > 3/2$.

\begin{thm}\label{thm:smoothness-strong-metric}
  Let $s > 3/2$ and $B_{\lambda} \in \mathcal{E}^{s}_{\lambda}(\ZZ)$. Then the expression
  \begin{equation*}
    G_{c}(h,k) = \int_{\Circle} \sprod{B_{c} h}{B_{c} k} \ud s\,,
  \end{equation*}
  defines a \emph{smooth and strong} Riemannian metric on $\mathcal{I}^{s}(\Circle,\RR^d)$.
\end{thm}
\begin{rem}
  For strong  metrics on a Lie group the invariance of the metric implies the geodesic and metric completeness of the space, see \cite[Lemma 5.2]{GMR2009}. This has been used in \cite{BEK2015}  to show completeness of the $H^s$-metric on $\Diff(\RR^d)$, see also \cite{BV2014} for integer orders on the diffeomorphism group of a general manifold. Unfortunately, there is no automatic analogue of this result in our situation. To prove the global well-posedness on the  space of regular curves additional assumptions on the dependence of the operator $A$ on the length will be necessary. In future work, we plan to follow this line of research and use a similar strategy as in \cite{Bru2015} for integer orders to obtain this result.
\end{rem}

\begin{proof}
  Since $B_{\lambda} \in \mathcal{E}^{s}_{\lambda}(\ZZ)$ and $s > 3/2$, the mapping
  \begin{equation*}
    \check G_{c} = B_{c}^{t} \circ M_{\abs{c^{\prime}}}\circ B_{c}
  \end{equation*}
  defines, for each $c\in \mathcal{I}^{s}$, a bounded isomorphism between $\HRd{s}$ and its dual $\HRd{-s}$. Thus we need only to show that the mapping
  \begin{equation*}
    c \mapsto \check G_{c}, \qquad \mathcal{I}^{s} \to \mathcal{L}(\HRd{s},\HRd{-s})
  \end{equation*}
  is smooth. Since transposition between Banach spaces is itself a bounded operator it follows that the transpose
  \begin{equation*}
    c \mapsto B_{c}^{t}, \quad \mathcal{I}^{s} \to \mathcal{L}(\LRd,\HRd{-s})
  \end{equation*}
  is smooth iff
  \begin{equation*}
    c \mapsto B_{c}, \quad \mathcal{I}^{s} \to \mathcal{L}(\HRd{s},\LRd)
  \end{equation*}
  is smooth, which is the case by Proposition~\ref{prop:smoothness-Ac}. Now, the mapping
  \begin{equation*}
    c \mapsto M_{\abs{c^{\prime}}} , \qquad \mathcal{I}^{s} \to \mathcal{L}(\LRd,\LRd),
  \end{equation*}
  is smooth for $s > 3/2$. Finally, since composition of bounded operators between Banach spaces is itself a bounded operator, it follows that the composition
  \begin{equation*}
    c \mapsto B_{c}^{t} \circ M_{\abs{c^{\prime}}}\circ B_{c}, \qquad \mathcal{I}^{s} \to \mathcal{L}(\HRd{s},\HRd{-s}),
  \end{equation*}
  is smooth.
\end{proof}

\begin{rem}\label{rem:strong-geodesic-spray}
  A smooth, strong Riemannian metric on a Hilbert manifold has a smooth spray (see~\cite{Lan1999} for instance). However, formula~\eqref{eq:geodesic-spray-imm} is no longer useful in that case because it is not clear that this expression extends to a smooth map from $\mathcal{I}^{s}$ to $\HRd{s}$. Following Lang~\cite[Proposition 7.2]{Lan1999}, we introduce the $H^{s}$ inner product on $\HRd{s}$:
  \begin{equation*}
    \llangle h,k \rrangle_{H^{s}} := \int_{\Circle} \sprod{\Lambda^{s} h}{\Lambda^{s} k} \ud \theta
  \end{equation*}
  so that
  \begin{equation*}
    G_{c}(h,k) = \llangle P_{c}h,k \rrangle_{H^{s}}, \quad \text{where} \quad P_{c} := \left(\Lambda^{-s} \circ B_{c}\right)^{t} \circ \Lambda^{-s} \circ M_{\abs{c^{\prime}}} \circ B_{c}.
  \end{equation*}
  Here, $Q^{t}$ is the transpose of a bounded operator
  \begin{equation*}
    Q: \HRd{s} \to \HRd{s}.
  \end{equation*}
  In that case, there is an alternative expression for the spray, which is nevertheless equivalent to~\eqref{eq:geodesic-spray-imm}. It is given implicitly by the formula
  \begin{equation*}
    \llangle S_{c}(h), P_{c}k \rrangle_{H^{s}} = \frac{1}{2} \llangle \left(D_{c,k}P_{c}\right)h,h \rrangle_{H^{s}} - \llangle\left(D_{c,h}P_{c}\right)h,k \rrangle_{H^{s}}.
  \end{equation*}
\end{rem}

\appendix

\section{Smooth curves in $\mathcal{S}^{r}(\ZZ)$}
\label{sec:smooth-curves}

The right framework to work with smooth maps on infinite dimensional (other than Banach spaces) are \emph{convenient vector spaces} (see~\cite{FK1988,KM1997}). Note that every Fr\'{e}chet space is a convenient vector space. We will need the following lemma about recognising smooth curves in convenient vector spaces.

\begin{lem}{\cite[4.1.19]{FK1988}}\label{lem:smooth-curves}
  Let $c: \RR \to E$ be a curve in a convenient vector space. Let $\mathcal{V} \subseteq E'$ be a point-separating subset of bounded linear functionals such that the bornology of $E$ has a basis of $\sigma(E, \mathcal{V})$-closed sets. Then the following are equivalent:
  \begin{enumerate}
    \item
          $c$ is smooth;
    \item
          There exist locally bounded curves $c^k: \RR \to E$ such that $\ell \circ c : \RR \to \RR$ is smooth with $(\ell \circ c)^{(k)} = \ell \circ c^k$, for each $\ell \in \mathcal{V}$.
  \end{enumerate}
\end{lem}

To apply this lemma to the space $\mathcal{S}^{r}(\ZZ)$ of Fourier multipliers we need to choose a suitable set $\mathcal{V}$ of linear functionals. This is accomplished in the following lemma.

\begin{lem}\label{lem:V-bornology}
  Let $r \in \RR$ and $\mathcal{V} = \{ \lambda \circ \operatorname{ev}_m \,:\, m \in \ZZ,\, \lambda \in \mathcal{L}(\CC^d)' \} \subset \mathcal{S}^{r}(\ZZ)'$. Then the bornology of $\mathcal{S}^{r}(\ZZ)$ has a basis consisting of $\sigma(\mathcal{S}^{r}(\ZZ), \mathcal{V})$-closed sets.
\end{lem}

\begin{proof}
  We can regard $\mathcal{V}$ as a subset of both $\mathcal{S}^{r}(\ZZ)'$ as well as $\ell^\infty(\ZZ, \mathcal{L}(\CC^d))'$. It is shown in \cite[4.1.21]{FK1988} that the bornology of $\ell^\infty(\ZZ, \mathcal{L}(\CC^d))$ has a basis of $\sigma(\ell^\infty(\ZZ, \mathcal{L}(\CC^d)), \mathcal{V})$-closed sets. We embed $\mathcal{S}^{r}(\ZZ)$ as a closed subspace into
  \begin{equation*}
    \iota: \mathcal{S}^{r}(\ZZ) \to \prod_{\alpha \in \NN} \ell^\infty(\ZZ, \mathcal{L}(\CC^d)),\,
    \left(\mathbf{a}(m)\right)_{m\in \ZZ} \mapsto \left(  \langle m \rangle ^{-(r-\alpha)} \Delta^{\alpha} \mathbf{a}(m) \right)_{m \in \ZZ,\, \alpha \in \NN}\,.
  \end{equation*}
  Because the bornology of the product is the product bornology it follows that a basis for the bornology on
  $\prod_{\alpha \in \NN} \ell^\infty(\ZZ, \mathcal{L}(\CC^d))$ is given by $\sigma(\prod_{\alpha} \ell^\infty,\mathcal{W})$-closed sets, where $\mathcal{W} = \bigcup_{\alpha \in \NN} \mathcal{V} \circ \operatorname{pr}_{\alpha}$ and $\operatorname{pr}_{\alpha}$ denotes the canonical projection onto the $\alpha$-th factor. Here we note that on $\mathcal{S}^{r}(\ZZ)'$ the set $\iota^\ast(\mathcal{W})$ is contained in the linear space of $\mathcal{V}$ and hence the bornology of $\mathcal{S}^{r}(\ZZ)$ has a basis consisting of $\sigma(\mathcal{S}^{r}(\ZZ), \mathcal{V})$-closed sets.
\end{proof}

\section{Square-root of an operator in $\mathcal{E}^{r}_{\lambda}(\ZZ)$}
\label{sec:square-root}

Since a positive definite Hermitian matrix has a unique positive square root, which depends smoothly on its coefficients, we can define formally the square root $B_{\lambda} = \op{\mathbf{a}(\lambda, m)^{1/2}}$ of an element $A_{\lambda} = \mathbf{a}_{\lambda}(D)$ in the class $\mathcal{E}^{r}_{\lambda}(\ZZ)$. In order to prove this we need the following lemma.

\begin{lem}{\cite[Lemma~4.8]{BEK2015}}\label{lem:hermite-sqrt-estimate}
  Let $a, b, x \in \mathcal{L}(\CC^d)$ be three matrices satisfying
  \begin{equation*}
    bx + xb = a\,,
  \end{equation*}
  with $b$ Hermitian and positive definite. Then
  \begin{equation*}
    \norm{x} \leq \sqrt{\frac d2} \norm{b^{-1}}\norm{a},
  \end{equation*}
  where $\norm{\cdot}$ denotes the Frobenius norm, i.e. $\norm{x} = \sqrt{\operatorname{tr} xx^\ast}$.
\end{lem}

The following lemma together with its proof is a generalisation of \cite[Lemma~4.7]{BEK2015} to our situation of one-parameter families of symbols.

\begin{lem}\label{lem:square-root}
  The positive square root of an operator in the class $\mathcal{E}^{r}_{\lambda}(\ZZ)$ belongs to the class $\mathcal{E}^{r/2}_{\lambda}(\ZZ)$. Conversely, the square of an operator in the class $\mathcal{E}^{r}_{\lambda}(\ZZ)$ belongs to the class $\mathcal{E}^{2r}_{\lambda}(\ZZ)$.
\end{lem}

\begin{proof}
  We will prove the estimate
  \begin{equation*}
    \norm{\partial^\beta_{\lambda} \Delta^{\alpha} \mathbf{b}(\lambda, m)} \lesssim \langle m\rangle ^{r/2-\alpha}\,,
  \end{equation*}
  which holds locally uniformly in $\lambda$, by induction over $\alpha + \beta$. If $\alpha + \beta = 0$ the statement is $\| \mathbf{b}(\lambda, m) \| \lesssim \langle m \rangle^{r/2}$. Assume that it has been proven for $\alpha + \beta \leq k$. Then let $\alpha + \beta = k+1$ and, omitting the arguments $(\lambda, m)$, we obtain using the product rule,
  \begin{equation*}
    \partial^\beta_{\lambda} \Delta^{\alpha} \mathbf{a} = \mathbf{b}\left( \partial^\beta_{\lambda} \Delta^{\alpha} \mathbf{b}\right) + \left( \partial^\beta_{\lambda} \Delta^{\alpha} \mathbf{b} \right) \mathbf{b}
    + \sum_{(i,j) \in X} \binom{\alpha}{i} \binom \beta j \partial^{j}_{\lambda} \Delta^i \mathbf{b}\, \partial^{\beta-j}_{\lambda} \Delta^{\alpha-i} \mathbf{b}\,,
  \end{equation*}
  with
  \begin{equation*}
    X = \set{(i,j) \,:\, 0 \leq i \leq \alpha,\, 0 \leq j \leq \beta,\, i + j < \alpha + \beta}\,.
  \end{equation*}
  Then by the induction assumption
  \begin{equation*}
    \norm{\partial^{j}_{\lambda} \Delta^i \mathbf{b}\, \partial^{\beta-j}_{\lambda} \Delta^{\alpha-i} \mathbf{b}}
    \lesssim \langle m \rangle^{r/2-i} \langle m \rangle^{r/2-\alpha+i} \lesssim \langle m \rangle^{r-\alpha}\,,
  \end{equation*}
  and hence we obtain via Lemma~\ref{lem:hermite-sqrt-estimate},
  \begin{equation*}
    \norm{\partial^\beta_{\lambda} \Delta^{\alpha} \mathbf{b}} \lesssim \langle m \rangle^{-r/2} \langle m \rangle^{r-\alpha}
    \lesssim \langle m \rangle^{r/2-\alpha}\,.
  \end{equation*}
  This completes the induction.
\end{proof}

\bibliographystyle{abbrv}
\bibliography{refs}
\end{document}